\documentclass[11pt]{amsart}
\usepackage{amsfonts}
\usepackage{amsmath,amsthm,latexsym,amssymb,graphicx}
\usepackage{graphicx,color,enumerate}

\numberwithin{equation}{section}

\newtheorem{theorem}{Theorem}

\newtheorem{proposition}{Proposition}

\newtheorem{remark}{Remark}

\numberwithin{theorem}{section}
\numberwithin{corollary}{section}
\numberwithin{lemma}{section}
\numberwithin{definition}{section}
\numberwithin{proposition}{section}
\numberwithin{remark}{section}

\newcommand{\dint}{\ds\int}
\newcommand{\ds}{\displaystyle}

\newcommand{\R}{\mathbb R}

\newcommand{\medint}{-\kern  -,375cm\int}

\begin{document}
\title[ ]{A sharp lower bound for some Neumann eigenvalues of the Hermite operator }
\author[B. Brandolini, F. Chiacchio, C. Trombetti]{B. Brandolini - F.
Chiacchio - C. Trombetti$^*$}
\thanks{$^*$ Dipartimento di Matematica e Applicazioni ``R. Caccioppoli'',
Universit\`{a} degli Studi di Napoli ``Federico II'', Complesso Monte S.
Angelo, via Cintia - 80126 Napoli, Italy; email: brandolini@unina.it;
francesco.chiacchio@unina.it; cristina@unina.it}
\maketitle

\begin{abstract}
This paper deals with the Neumann eigenvalue problem for the Hermite operator defined in a convex, possibly unbounded, planar domain $\Omega$, having one axis of symmetry passing through the origin. We prove a sharp lower bound for the first eigenvalue $\mu_1^{odd}(\Omega)$ with an associated eigenfunction odd with respect to the axis of symmetry. Such an estimate involves the first eigenvalue of the corresponding one-dimensional problem.  As an immediate consequence, in the class of domains for which $\mu_1(\Omega)=\mu_1^{odd}(\Omega)$, we get  an explicit lower bound for the difference between $\mu(\Omega)$ and the first Neumann eigenvalue of any strip.
\end{abstract}

\section{Introduction}

In this paper we study the following Neumann eigenvalue problem
\begin{equation}
\left\{
\begin{array}{ll}
-\mathrm{div}\left( \exp \left( -\frac{x^{2}+y^{2}}{2}\right) \nabla
u\right) =\mu \exp \left( -\frac{x^{2}+y^{2}}{2}\right) u & \mbox{in}\>\Omega
\\
&  \\
\dfrac{\partial u}{\partial \nu }=0 & \mbox{on}\>\partial \Omega,
\end{array}
\right.  \label{problem}
\end{equation}
where $\Omega $ is a possibly unbounded, smooth domain in $\mathbb{R}
^{2} $ and $\nu$ is the outward normal to
$\partial \Omega $. As in the case of the Neumann Laplacian it is easily seen that the
lowest eigenvalue of problem \eqref{problem} is zero, the eigenfunction being any constant. Eigenfunctions $u$ corresponding to higher
eigenvalues must satisfy the orthogonality condition
\begin{equation*}
\int_{\Omega }u\,d\gamma _{2}=0,
\end{equation*}
where $d\gamma _{2}$ stands for the standard normal Gaussian measure, that is
\begin{equation*}
d\gamma _{2}=d\gamma_x \otimes d\gamma_y=\frac{1}{\sqrt{2\pi} }\exp \left( -\frac{x^{2}}{2}\right) dx \otimes \frac{1}{\sqrt{2\pi} }\exp \left( -\frac{y^{2}}{2}\right) dy.
\end{equation*}
Clearly the equation in \eqref{problem} can be rewritten as follows
\begin{equation*}
-\Delta u+x\cdot \nabla u=\mu u,
\end{equation*}
where at the left hand side the classical Hermite operator appears.
When $\Omega =\mathbb{R}^{2}$, all the eigenvalues of \eqref{problem} are
known and corresponding eigenfunctions are the Hermite polynomials (see,
e.g., \cite{CH}). When $\Omega \subsetneqq \mathbb{R}^{2},$ much more less is
known about the spectral properties of this operator.

In the case of  Dirichlet homogeneous boundary condition a Faber-Krahn type inequality
has been established by Ehrhard in 1986. In \cite{E} (see also \cite{BCF})  he actually proved that, among all domains in $\R^N$ having prescribed Gaussian measure, the half-space achieves the smallest eigenvalue. A sharp inequality concerning the ratio between the first two eigenvalues (the so-called PPW estimate) is contained in \cite{BL}.

When a Neumann homogeneous boundary condition is
prescribed, in \cite{CdB} a Szeg\"{o}-Weinberger type inequality is derived;
more precisely the authors proved that
\begin{equation*}
\mu _{1}(\Omega )\leq \mu _{1}(\Omega ^{\sharp }),
\end{equation*}
where $\Omega $ is a smooth domain of $\mathbb{R}^{N},$ symmetric with
respect to the origin and possibly unbounded, and $\Omega ^{\sharp }$ is the
ball, centered at the origin, with the same Gaussian measure as $\Omega $.

In this paper we consider the class of planar convex sets having an axis of symmetry
through the origin, say for instance the $y$-axis, and we denote by $\mu _{1}^{odd}(\Omega )$ the lowest Neumann eigenvalue  with
a corrisponding eigenfunction odd with respect  to the axis of symmetry. We prove a lower bound for $\mu _{1}^{odd}(\Omega )$ in the same spirit of the  celebrated Payne-Weinberger estimate for the
first nontrivial eigenvalue $\mu _{1}^{\Delta }(\Omega )$\ of the Laplacian.
In \cite{PW} (see also \cite{Be}) the authors proved that if $\Omega $
is a bounded convex domain in $\mathbb{R}^{N}$ with diameter $d$, then
\begin{equation}
\mu _{1}^{\Delta }(\Omega )\geq \frac{\pi ^{2}}{d^{2}}.  \label{PW_est}
\end{equation}
This result is asymptotically sharp, since $ \pi ^{2}/d^{2} $ is the
first nontrivial Neumann eigenvalue of the  one-dimensional Laplacian in $
(-d/2,d/2).$ Instead of $ \pi^2/d^2 $ our estimate for $\mu
_{1}^{odd}(\Omega )$ involves the first nontrivial
eigenvalue $\mu _{1}(a,b)$ of the following one-dimensional problem
\begin{equation*}
\left\{
\begin{array}{ll}
-v^{\prime \prime }+xv^{\prime }=\mu v & \mbox{in}\>(a,b) \\
&  \\
v^{\prime }(a)=v^{\prime }(b)=0, &
\end{array}
\right.
\end{equation*}
where $-\infty \leq a<b\leq \infty $ \ are related to the geometry of $
\Omega .$ As well-known the following variational characterization holds
\begin{equation}
\mu _{1}(a,b)=\min_{\int_{a}^{b}z\,d\gamma _{x}=0}\frac{\displaystyle
\int_{a}^{b}(z^{\prime })^{2}d\gamma _{x}}{\displaystyle\int_{a}^{b}z^{2}d
\gamma _{x}}.  \label{min}
\end{equation}

Our first main result is the following.

\begin{theorem}
\label{T1} Let $\Omega$ be a convex, bounded domain in $\R^2$, symmetric with respect to the $y$-axis. Let $a>0$ and let $p,-q:(-a,a)\rightarrow \mathbb{R}$ be concave,
even functions such that
\begin{equation*}
\Omega =\{(x,y)\in \mathbb{R}^{2}:\>-a<x<a,\>q(x)<y<p(x)\};
\end{equation*}
 then
\begin{equation}
\mu _{1}^{odd}(\Omega )\geq \mu _{1}(-a,a),  \label{estimate}
\end{equation}
where $\mu _{1}(-a,a)$ is defined in \eqref{min}. Equality sign holds in
\eqref{estimate} for every rectangle $(-a,a)\times (-b,b)$  with $b>0$.
\end{theorem}

The case of unbounded sets is addressed in the next theorem. We will implicitly suppose that $\Omega$ is not a vertical strip, being this case trivial.
\begin{theorem}
\label{T2} Let $\Omega$ be a $C^2$, convex, unbounded domain in $\R^2$, symmetric with respect to the $y$-axis. Let $a \in (0,+\infty]$ and let $p:(-a,a)\rightarrow \mathbb{R}$ be a concave, even, $C^2$ function
such that
\begin{equation*}
\Omega =\{(x,y)\in \mathbb{R}^{2}:\>-a<x<a,\>y<p(x)\} .
\end{equation*}
Assume that $\Omega$ satisfies a uniform interior sphere condition. Then \eqref{estimate} holds true. Equality sign holds in \eqref{estimate} for every domain of the type $(-a,a)\times (-\infty,b)$ with $b\in \R$.
\end{theorem}
As for the classical Neumann Laplacian, the convexity assumption cannot be removed. Indeed, consider a set $\Omega _{\epsilon }$
consisting of two squares connected by a corridor having width $\epsilon $ in
such a way that the set is symmetric with respect to the $y$-axis. It is easy
to verify that, as $\epsilon $ goes to zero, $\mu _{1}^{odd}(\Omega _{\epsilon })$
goes to zero.

Very recently we found the paper \cite{AN} where the authors, among other things, prove that  if $\Omega$ is a bounded  convex subset of $\R^n$ ($n \ge 2$), then
\begin{equation}
\mu_1(\Omega) \ge \mu_1\left(-\frac{d}{2},\frac{d}{2}\right), \label{andrews-ni}
\end{equation}
where $d$ denotes the diameter of $\Omega$ (see also \cite{Ni}, \cite{AC}). Estimate \eqref{andrews-ni} cannot be compared with ours \eqref{estimate} which involves $\mu_1^{odd}(\Omega)$, the maximum distance $a$ of $\partial\Omega$ from the axis of symmetry, and holds true also for unbounded domains. In a forthcoming paper \cite{BCHT} we will prove that it is possible to ``pass to the limit'' in \eqref{andrews-ni} as $d \to \infty$ and hence prove that  when $\Omega$ is a convex, unbounded domain of $\R^n$ ($n \ge 2$), then
$$
\mu_1(\Omega) \ge \mu_1(\R)=1.
$$
%Our estimate \eqref{estimate} turns out to be tight, and indeed achieved, since we reduce  \eqref{problem} to a one-dimensional problem of the same type. We point out that in \cite{Ni}  Lei Ni, among other things, extends Payne-Weinberger inequality \eqref{PW_est} to a family of $n$-dimensional eigenvalue problems including \eqref{problem}.
%In some special cases we can compare our results with the one by Ni. This happens, for instance, in the class $\mathcal{K}$ of domains $\Omega$ satisfying assumptions of Theorem \ref{T1}  for which  $\mu_1^{odd}(\Omega)=\mu_1(\Omega)$. Clearly ${\mathcal K}$ is not empty since  it contains any disk centred at the origin (see \cite{CdB}, Lemma 4.1) and any square $(-l,l)^2$ ($l>0$). If $\Omega \in {\mathcal K}$, Ni's result states that
%$$
%\mu_1(\Omega)\geq \frac{\pi ^{2}}{\mathrm{diam}(\Omega )^{2}},
%$$
%which is weaker than \eqref{estimate}, since, as it is easy to verify, $\mu_1(-a,a)>\dfrac{\pi^2}{4a^2} \ge \dfrac{\pi^2}{\mathrm{diam}(\Omega )^{2}}$ for any $a>0$. Moreover, we are able to handle unbounded convex sets for which Ni's inequality does not apply.
%
%Finally whenever $\Omega$ belongs to  the class $\mathcal{K}$, denoted by $S_a=\{(x,y)\in \R^2: \> |x|<a\}$, an immediate consequence of \eqref{estimate} is
%$$
%\mu_1(\Omega)-\mu_1(S_a)=\mu_1(\Omega)-1 \ge \lambda_1(-a,a)
%$$
%where $\lambda_1(-a,a)$ stands for the first Dirichlet eigenvalue of the Hermite operator in $(-a,a)$ (see Section 3).

\section{The case of bounded domains}
In order to prove Theorem \ref{T1} we first decompose the domain $\Omega$ into convex subdomains, symmetric with respect to the $y$-axis, having small Gaussian measure and width. Then we prove a lower bound for a class of Sturm-Liouville problems, which is somehow a stronger version of the one-dimensioanl analogue to \eqref{estimate}. Finally we use the boundedness of an eigenfunction corresponding to $\mu_1^{odd}(\Omega)$, together with its first and second derivatives, to pass from one to two dimensions.

\begin{proof}[Proof of Theorem \protect\ref{T1}]
\textbf{Step 1: Decomposition in horizontal strips.}  By
approximation arguments (see, for example,  \cite{C} and also \cite{H} p. 35), we may
always assume  $p,q \in C^2(-a,a)$ and that there exist $c_0,c_1>0$ such that
\begin{equation}  \label{p'}
c_0 \le |p^{\prime }(x)|,|q^{\prime }(x)|\le c_1
\end{equation}
for every $x \in (-a,a)$. Let us consider the set of straight-lines parallel
to the $x$-axis and contained in the half-plane $\{y>p(a)\}$; one of them
divides $\Omega^+= \Omega \cap \{y>p(a)\}$ into two convex subdomains with
the same Gaussian measure over each of which $u$ has zero mean value with
respect to $d\gamma_2 $. Repeating this process $n$ times we get
\begin{equation*}
\gamma_2(\Omega_k)= \frac{\gamma_2(\Omega^+)}{2^n}, \quad
\Omega^+=\bigcup_{k=1}^n \Omega_k, \quad \Omega_k \>\mbox{convex}, \quad
\int_{\Omega_k}u d\gamma_2 =0.
\end{equation*}
Clearly, by construction,
\begin{equation*}
\Omega_k=\{(x,y)\in \mathbb{R}^2: \> -a_k\le x \le a_k, d_k \le y \le
p_k(x)\}.
\end{equation*}

\bigskip
\begin{figure}[h]
\centering
\includegraphics[scale=0.6]{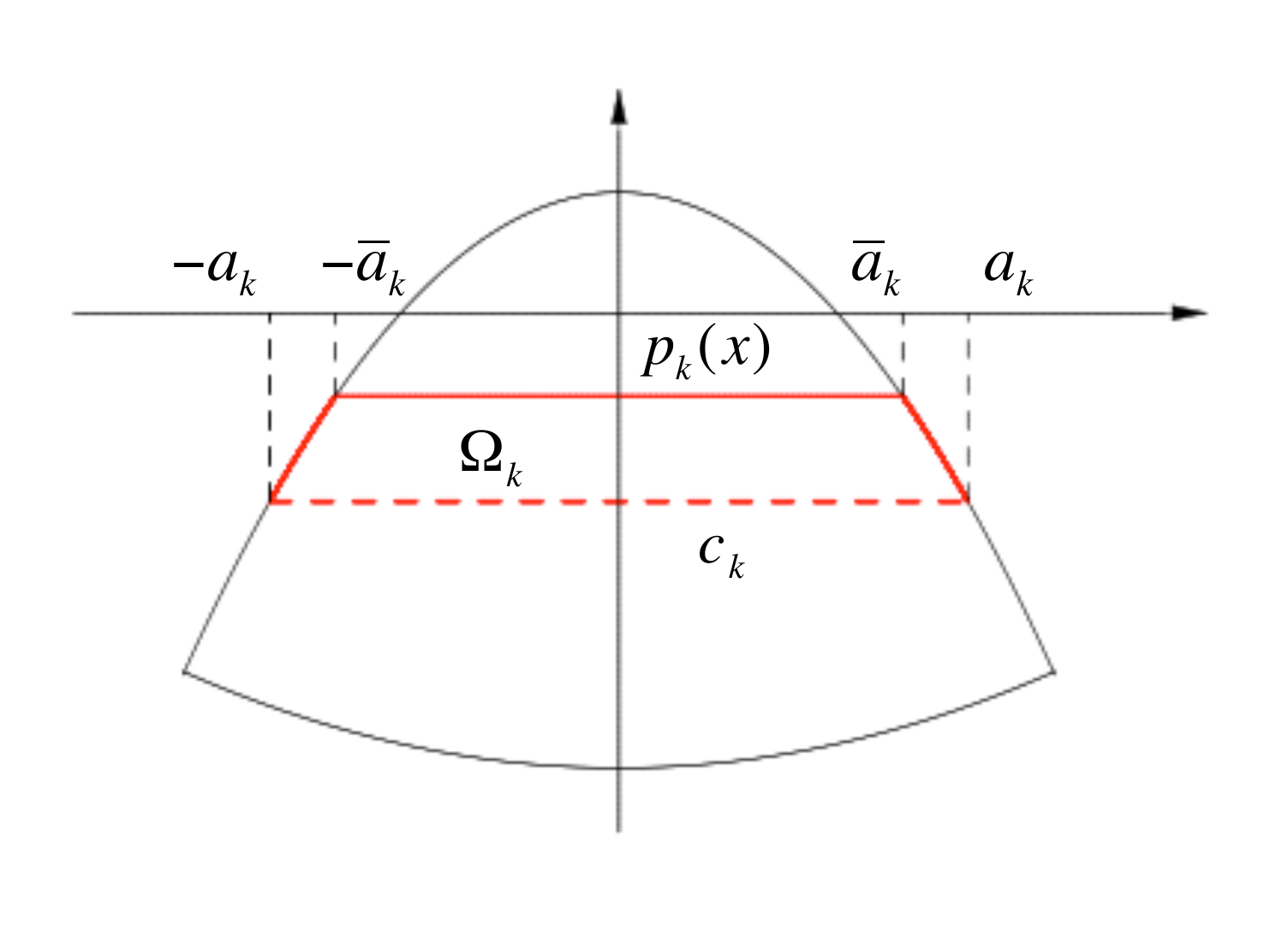} \begin{caption}
{}
\end{caption}
\end{figure}

\noindent Let us fix $\epsilon \in (0,\overline \epsilon)$, where the value
of $\overline \epsilon$ will be specified later; for sufficiently large $n$,
we have that
\begin{equation}  \label{sp}
\gamma_2(\Omega_k) < \epsilon \quad \mbox{and} \quad p_k(x)-d_k < \epsilon,
\quad 1 \le k \le n.
\end{equation}
\textbf{Step 2: A one-dimensional auxiliary problem.} Set
$$\phi_k(x)=\displaystyle\int_{d_k}^{p_k(x)}d\gamma_y;$$
because of the concavity of $p$ and recalling the definition of $p_k$ and $d_k$  in Step 1, we have that $\phi_k(x)$ is a convex function.  Let
\begin{equation}\label{lambda}
\bar \lambda_k=\min\left\{\frac{\dint_{-a_k}^{a_k}(z')^2\phi_kd\gamma_x}{\dint_{-a_k}^{a_k}z^2\phi_kd\gamma_x}:\> z \in C^1(-a_k,a_k), \>\int_{-a_k}^{a_k}z\phi_kd\gamma_x=0\right\}.
\end{equation}
Then a function $v_k$ which realizes the minimum in \eqref{lambda} must satisfy the condition
\begin{equation}\label{zeromean}
\dint_{-a_k}^{a_k}v_k\phi_kd\gamma_x=0
\end{equation}
and the following Sturm-Liouville problem
\begin{equation*}
\left\{
\begin{array}{ll}
-\left[ v_k^{\prime }\phi_k \gamma _{x}\right] ^{\prime }=\bar\lambda_k v_k\phi_k
\gamma _{x} & \mbox{in}\>(-a_k,a_k) \\
&  \\
v_k^{\prime }(-a_k)\phi_k (-a_k)=v_k^{\prime }(a_k)\phi_k (a_k)=0. &
\end{array}
\right.
\end{equation*}
We differentiate with respect to $x$ and introduce the new variable $
w_k=v_k^{\prime}\phi_k^{1/2}$. The function $w_k$ satisfies the following problem
\begin{equation*}
\left\{
\begin{array}{ll}
-w_k^{\prime \prime }+xw_k^{\prime }+w_k\left[ -\frac{1}{2}\frac{\phi_k
^{\prime \prime }}{\phi_k }+\frac{3}{4}\left( \frac{\phi_k ^{\prime }}{
\phi_k }\right) ^{2}-\frac{1}{2}x\frac{\phi_k^{\prime }}{\phi_k }\right]
=(\bar\lambda_k -1)w_k & \mbox{in}\> (-a_k,a_k) \\
&  \\
w(-a_k)=w(a_k)=0. &
\end{array}
\right.
\end{equation*}
Multiplying by $w_k(x)\exp \left( -\frac{x^{2}}{2}\right) $ and integrating
over $(-a_k,a_k)$ we get
\begin{eqnarray*}  \label{w}
I_1+I_2&=&\int_{-a_k}^{a_k}(w_k^{\prime})^2d\gamma _{x} +
\int_{-a_k}^{a_k}\left(\frac{w_k }{\phi_k}\right)^2\left[ -\frac{1}{2}\phi_k
\phi_k ^{\prime \prime } +\frac{3}{4} \left( \phi_k ^{\prime } \right) ^{2}-
\frac{1}{2}x \phi_k\phi_k^{\prime }\right] d\gamma _{x} \\
&=&(\bar\lambda_k -1)\int_{-a_k}^{a_k}w_k^{2}d\gamma _{x}.  \notag
\end{eqnarray*}
By construction, $p_k$ is concave in $(-a_k,a_k)$ and constant in $
(-\overline a_k,\overline a_k)$; then, recalling the definition of $\phi_k$,
we have
\begin{eqnarray*}  \label{q}
I_2&=& \int_{\overline a_k}^{a_k}\left(\frac{w_k }{\phi_k}\right)^2\left[
\phi_k e^{-\frac{p_k^2}{2}}\left(-p_k^{\prime \prime }+p_k(p_k^{\prime })^2
-x p_k^{\prime }e^{-\frac{p_k^2}{2}}\right)+\frac{3}{2}(p_k^{\prime})^2e^{-p_k^2}\right]d\gamma_x \\
&\ge & \int_{\overline a_k}^{a_k}\left(\frac{w_k }{\phi_k}\right)^2\left[
\phi_k e^{-\frac{p_k^2}{2}}\left(p_k(p_k^{\prime})^2-x  p_k^{\prime }e^{-
\frac{p_k^2}{2}}\right)+\frac{3}{2}(p_k^{\prime})^2e^{-p_k^2}\right]d\gamma_x
.  \notag
\end{eqnarray*}
It can be easily seen that there exists a positive constant $L$ independent
of $n$ and $k$ such that
\begin{equation*}
\left|p_k(x)(p_k^{\prime }(x))^2-x p_k^{\prime }(x)e^{-\frac{p_k(x)^2}{2}}\right|
\le L \qquad \forall \> x \in (\overline a_k,a_k);
\end{equation*}
then, choosing $\overline \epsilon =\frac{3}{2}\frac{c_0^2}{L}$, by
\eqref{p'} and \eqref{sp}, the integral $I_2$ is bounded from below by a
positive constant. Therefore
\begin{equation*}
\bar\lambda_k -1\geq \frac{\displaystyle\int_{-a_k}^{a_k} (w_k^{\prime})^2d\gamma
_{x}}{\displaystyle\int_{-a_k}^{a_k}w_k^{2}d\gamma _{x}}\geq \lambda
_{1}(-a_k,a_k),  \label{ge}
\end{equation*}
where $\lambda_1(-a_k,a_k)$ is the first Dirichlet eigenvalue of the
one-dimensional Hermite operator in $(-a_k,a_k)$. Since it can be easily
seen that
\begin{equation}\label{monotone}
\mu_1(-a_k,a_k)=\lambda_1(-a_k,a_k)+1 \ge \lambda_1(-a,a)+1=\mu_1(-a,a),
\end{equation}
we get
\begin{equation}\label{one}
\bar\lambda_k \ge \mu_1(-a,a).
\end{equation}
\textbf{Step 3: Estimates by a dimension reduction process.} Let $u$ be an eigenfunction corresponding to $\mu_1^{odd}(\Omega)$ and let $M>0$ be an
upper bound for the absolute values of $u$ and its first and second
derivatives. By Lagrange theorem we get
\begin{eqnarray}\label{1}
\Big|\int_{\Omega_k} \left(\frac{\partial u}{\partial x}\right)^2& d\gamma_2&
 -  \int_{-a_k}^{a_k}\left(\frac{\partial u}{\partial x}(x,d_k)\right)^2
\phi_k(x) d\gamma_x \Big|
\\
&\le& 2 M^2 \int_{-a_k}^{a_k} d\gamma_x
\left(\int_{d_k}^{p_k(x)}(y-d_k)d\gamma_y\right)\notag
\le 2M^2 \gamma_2(\Omega_k) \epsilon. \notag
\end{eqnarray}
Analogously
\begin{equation}\label{2}
\left| \int_{\Omega_k}u^2 d\gamma_2 -\int_{-a_k}^{a_k}
u(x,d_k)^2\phi_k(x)d\gamma_x \right| \le 2 M^2 \gamma_2(\Omega_k) \epsilon
\end{equation}
and
\begin{equation}\label{3}
\left|\int_{\Omega_k}u d\gamma_2 -\int_{-a_k}^{a_k} u(x,d_k) \phi_k(x)
d\gamma_x \right| \le M \gamma_2(\Omega_k) \epsilon.
\end{equation}
Since the function $u(x,c_k)-\frac{1}{\gamma_2(\Omega_k)}\dint_{-a_k}^{a_k}u(x,c_k)\phi_k(x)d\gamma_x$ satisfies condition \eqref{zeromean}, from \eqref{one}, \eqref{1}, \eqref{2}, \eqref{3} and \eqref{monotone} we deduce
\begin{eqnarray*}
&\displaystyle\int_{\Omega _{k}}&\left\vert \nabla u\right\vert ^{2}d\gamma
_{2} \geq \int_{\Omega _{k}}\left( \frac{\partial u}{\partial x}\right)
^{2}d\gamma _{2} \\
&\geq& \int_{-a_{k}}^{a_{k}}\left( \frac{\partial u}{\partial x}
(x,c_{k})\right) ^{2}\phi _{k}(x)d\gamma_x-2M^{2}\epsilon \gamma _{2}(\Omega
_{k}) \\
&\geq& \mu _{1}(-a_{k},a_{k}) \int_{-a_{k}}^{a_{k}}\left[ u(x,c_{k})-\frac{1
}{\gamma _{2}(\Omega _{k})}\int_{-a_{k}}^{a_{k}}u(x,c_{k})\phi
_{k}(x)d\gamma _{x}\right] ^{2}\phi _{k}(x)d\gamma _{x}-2M^{2}\epsilon
\gamma _{2}(\Omega _{k}) \\
&\ge& \mu _{1}(-a ,a ) \left[ \int_{-a_{k}}^{a_{k}}u(x,c_{k})^{2}\phi
_{k}(x)d\gamma _{x}-\frac{1}{\gamma _{2}(\Omega _{k})}\left(
\int_{-a_{k}}^{a_{k}}u(x,c_{k}) \phi _{k}(x)d\gamma _{x}\right) ^{2}\right]
-2M^{2}\epsilon \gamma _{2}(\Omega _{k}) \\
&\geq&\mu _{1}(-a ,a ) \left[ \int_{\Omega _{k}}u^{2}d\gamma
_{2}-2M^{2}\epsilon \gamma _{2}(\Omega _{k})-\frac{1}{\gamma _{2}(\Omega
_{k})}\left( \int_{\Omega _{k}}ud\gamma _{2}+M\epsilon \gamma _{2}(\Omega
_{k})\right) ^{2}\right] -2M^{2}\epsilon \gamma _{2}(\Omega _{k}) \\
&=& \mu _{1}(-a ,a ) \int_{\Omega _{k}}u^{2}d\gamma _{2}- \mu _{1}(-a,a)
M^2(2+\epsilon)\epsilon \gamma _{2}(\Omega _{k}) -2M^{2}\epsilon
\gamma_{2}(\Omega _{k}).
\end{eqnarray*}
Summing over $k$ we get
\begin{equation*}
\int_{\Omega^+ }\left\vert \nabla u\right\vert ^{2}d\gamma _{2}\geq \mu
_{1}(-a,a) \int_{\Omega^+ }u^{2}d\gamma _{2}- \mu _{1}(-a,a)
M^2(2+\epsilon)\epsilon \gamma _{2}(\Omega^+ ) -2M^{2}\epsilon
\gamma_{2}(\Omega^+ ).
\end{equation*}
Finally as $\epsilon $ goes to $0^{+}$ we deduce
\begin{equation*}
\int_{\Omega^+ }\left\vert \nabla u\right\vert ^{2}d\gamma _{2}\geq \mu
_{1}(-a,a) \int_{\Omega^+ }u^{2}d\gamma _{2}.
\end{equation*}
Finally we can repeat the same arguments for $\Omega^-=\Omega \cap \{y
<q(a)\}$ after a reflection about the $x$-axis, obtaining
\begin{equation*}
\int_{\Omega^- }\left\vert \nabla u\right\vert ^{2}d\gamma _{2}\geq \mu
_{1}(-a,a) \int_{\Omega^- }u^{2}d\gamma _{2}
\end{equation*}
and the thesis follows.
\end{proof}

\section{The case of unbounded domains}
The arguments contained in the proof of Theorem \ref{T1} cannot be used to treat the case of unbounded domains, since in general an eigenfunction corresponding to $\mu_1^{odd}$ is not bounded. Consider, for instance, $\Omega=\{ (x,y) \in \R^2:  y<0\}$; then $\mu_1^{odd}(\Omega)=\mu_1(\Omega)=1$ and a corresponding eigenfunction is $u(x,y)=x$. To overcome this difficulty  we will
consider a sequence of bounded sets $\Omega_n$ satisfying the assumptions of
Theorem \ref{T1}, invading  $\Omega$ (see Figure 2) and then we pass to the limit in the eigenvalue estimate. A key point here is an extension result provided in Step 2  that could be of interest by its own. We cannot use the results already available in literature (see, for instance, \cite{C}, \cite{FP}), since we need a uniformly bounded sequence of extension operators $P_{\Omega_n}: H^1(\Omega_n,d\gamma_2) \to H^1(\R^2,d\gamma_2)$.
\bigskip
\begin{figure}[h]
\centering
\includegraphics[scale=0.6]{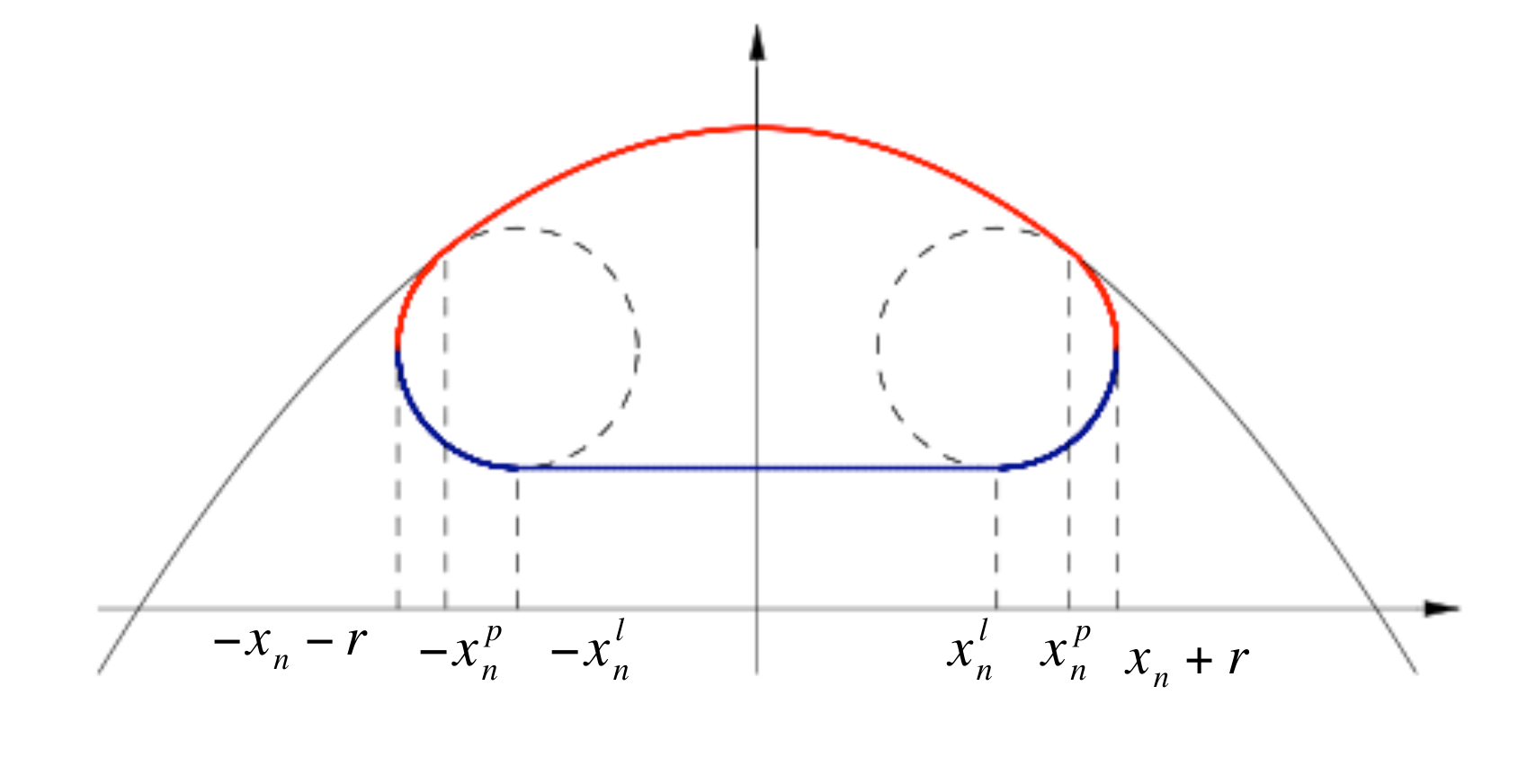} \begin{caption}
{}
\end{caption}
\end{figure}

\begin{proof}[Proof of Theorem \protect\ref{T2}]
\textbf{Step 1: A sequence of domains $\Omega_n$ invading $\Omega$.}

We distinguish two cases:

$\begin{array}{ll}
a)& \displaystyle\lim_{x \to a^-} p(x)=-\infty; \\
b)&\displaystyle\lim_{x \to a^-}p(x) \in \R.
\end{array}
$

Consider first case $a)$.
Let
\begin{equation*}
G_n=\{(x,y)\in\Omega: \> y > -n\}, \quad n\ge \tilde n=[-p(0)]+1,
\end{equation*}
where $[x]$ stands for the integer part of the real number $x$.

Note that $G_n \ne \emptyset$ for every $n \ge \tilde n$. In order to remove
the wedges at the bottom of $G_n$, we consider, for every $n \ge \tilde n$,
two equal disks $D_n^\pm$ of radius $r$ (whose value, independent of $n$,
will be specified later) centered at $(\pm x_n,-n+r)$, contained in $G_n$,
tangent both to $\partial\Omega$ and to $\{y=-n\}$ at the points $(\pm
x^p_n,p(x_n^p))$ and $(\pm x_n^l,-n)$, respectively. Finally, let
\begin{equation*}
p_n(x)=\left\{
\begin{array}{ll}
p(x) & -x_n^l \le x \le x_n^p \\
-n+r+\sqrt{r^2-(x+x_n)^2} & -x_n-r<x<-x_n^p \\
-n+r+\sqrt{r^2-(x-x_n)^2} & x_n^p<x<x_n+r
\end{array}
\right.
\end{equation*}
and
\begin{equation*}
q_n(x)=\left\{
\begin{array}{ll}
-n & -x_n^p \le x \le x_n^l \\
-n+r-\sqrt{r^2-(x+x_n)^2} & -x_n-r<x<-x_n^l \\
-n+r-\sqrt{r^2-(x-x_n)^2} & x_n^l<x<x_n+r.
\end{array}
\right.
\end{equation*}
Set
\begin{equation*}
\Omega_n=\{(x,y)\in \Omega:\> -x_n-r<x<x_n+r,\> q_n(x)<y<p_n(x)\}
\end{equation*}
(see Figure 2). Clearly $\Omega_n$ is a sequence of bounded, smooth, nested
sets whose union coincides with $\Omega$.
Since   $\Omega$
satisfies a uniform interior sphere condition we may choose
\begin{equation*}  \label{r}
r=\tilde r=\frac{1}{2} \inf \displaystyle{\frac {1}{|k(x,y)|}}>0,
\end{equation*}
where $k(x,y)$ stands for the curvature of $\partial \Omega$ at a generic
point $(x,y).$ Let $\overline \Omega_n=\Omega_n + B(0, \tilde r)$ be the exterior
parallel set of $\Omega_n$ relative to $B(0,\tilde r)$, that is the union of all closed disks of radius $\tilde r$ whose centres lie in $\Omega_n$. Denote by $\underline \Omega_n$
 the interior parallel set of $\Omega_n$ relative to $B(0,\tilde r)$, that is the
union of the centres of all disks with radius $\tilde r$ lying entirely in $\Omega_n $.

\textbf{Step 2: A uniformly bounded sequence of extension operators from $H^1(\Omega_n,d\gamma_2)$ onto $H^1(\R^2,d\gamma_2)$.}

Let $u_n$ be an eigenfunction corresponding to $\mu_1^{odd}(\Omega_n)$. We want to extend $u_n$ to $\overline \Omega_n$ by reflection along the normal
to $\partial \Omega_n$. Let $\Phi_n: (x,y) \in
\Omega_n\setminus\underline\Omega_n \to (x_e,y_e) \in \overline \Omega_n
\setminus \Omega_n$ be such a reflection. $\Phi_n$ is a one-to-one map and,
denoted by
$$x_m=\frac{x+x_e}{2},$$
 it holds
\begin{equation*}
J_{\Phi_n}(x,y)= \frac{\partial(x_e,y_e)}{\partial(x,y)}= \left\{
\begin{array}{ll}
\dfrac{1+(p_n^{\prime }(x_m))^2+p_n^{\prime \prime }(x_m)(y-p_n(x_m))}{
-1-(p_n^{\prime }(x_m))^2+p_n^{\prime \prime }(x_m)(y-p_n(x_m))}, \quad &
y\ge -n+\tilde r \\
&  \\
\dfrac{1+(q_n^{\prime }(x_m))^2+q_n^{\prime \prime }(x_m)(y-q_n(x_m))}{
-1-(q_n^{\prime }(x_m))^2+q_n^{\prime \prime }(x_m)(y-q_n(x_m))}, & y<-n+\tilde r.
\end{array}
\right.
\end{equation*}
We explicitly observe that the quantities $p_n^{\prime \prime
}(x_m)(y-p_n(x_m))$ and $q_n^{\prime \prime }(x_m)(y-q_n(x_m))$ are
nonnegative; so it is easy to verify that
$$|J_{\Phi_n}| \ge 1 \quad \mathrm{in} \> \underline \Omega_n.$$
 On the other hand, if $y \ge -n+\tilde r$,
 $$p_n(x_m)-y=\frac{\overline r}{\sqrt{1+(p_n^{\prime }(x_m))^2}},$$
  where
  \begin{equation}\label{r}
  \overline r=\mathrm{dist}\left((x,y),\partial\Omega\right) \in [0,\tilde r].
  \end{equation}
Then, from \eqref{r} we
deduce that
\begin{equation*}
0 \le p_n^{\prime \prime }(x_m)(y-p_n(x_m)) \le \frac{1}{2}(1+(p_n^{\prime
}(x_m))^2).
\end{equation*}
Thus $|J_{\Phi_n}|\le 3$ whenever $y \ge -n+\tilde r$.

\noindent Analogously one can treat
the case $y<-n+\tilde r$ obtaining
\begin{equation}  \label{jac}
1 \le| J_{\Phi_n}(x,y)| \le 3, \quad \forall (x,y)\in \underline\Omega_n.
\end{equation}
Define
\begin{equation}  \label{cris}
\overline u_n(x_e,y_e) = u_n(\Phi_n^{-1}(x_e,y_e)) \quad \forall
\>(x_e,y_e)\in \overline\Omega_n \setminus \Omega_n.
\end{equation}
If $(x_e,y_e)=\Phi_n(x,y)$ with $y \ge -n+\tilde r$, we have
\begin{eqnarray*}
\frac{\partial \overline u_n(x_e,y_e)}{\partial x_e}&=& \frac{\partial
u_n(x,y)}{\partial x}\left(\frac{2}{1+(p_n^{\prime }(x_m))^2+p_n^{\prime
\prime }(x_m)(y-p_n(x_m))}-1\right) \\
&&+\frac{\partial u_n(x,y)}{\partial y}\frac{2p_n^{\prime }(x_m)}{
1+(p_n^{\prime }(x_m))^2+p_n^{\prime \prime }(x_m)(y-p_n(x_m))} \\
\frac{\partial \overline u_n(x_e,y_e)}{\partial y_e}&=& \frac{\partial
u_n(x,y)}{\partial x}\frac{2p_n^{\prime }(x_m)}{1+(p_n^{\prime
}(x_m))^2+p_n^{\prime \prime }(x_m)(y-p_n(x_m))} \\
&&+\frac{\partial u_n(x,y)}{\partial y}\left(\frac{2p_n^{\prime }(x_m)^2}{
1+(p_n^{\prime }(x_m))^2+p_n^{\prime \prime }(x_m)(y-p_n(x_m))}-1\right).
\end{eqnarray*}
Analogous equalities hold whenever $y<-n+\tilde r$; then we can easily deduce that
\begin{equation}  \label{energy}
\left( \frac{\partial \overline u_n(x_e,y_e)}{\partial x_e}\right)^2 + \left(
\frac{\partial \overline u_n(x_e,y_e)}{\partial y_e}\right)^2 \le 2 \left[
\left(\frac{\partial u_n(x,y)}{\partial x}\right)^2+\left(\frac{\partial
u_n(x,y)}{\partial y}\right)^2\right].
\end{equation}
Now let $\theta_n \in C_0^\infty(\mathbb{R}^2)$ be such that $0 \le \theta_n
\le 1$ in $\mathbb{R}^2$, $\theta_n=1$ on $\Omega_n$, $\theta_n=0$ in $
\mathbb{R}^2\setminus \overline \Omega_n$ and $|D\theta_n| \le C$, with $C$
independent of $n$ and dependent only on $r$. Set
\begin{equation*}
\tilde u_n=\theta_n \overline u_n.
\end{equation*}
To go on we claim that there exists a positive constant $C$, independent of $n$, such that
\begin{equation}  \label{exp}
\exp\left(\frac{-||\Phi_n (x,y)||^2}{2}+\frac{x^2+y^2}{2}\right) \le C \quad \forall \> (x,y) \in \underline\Omega_n.
\end{equation}
 Indeed a straightforward computation yields that, if $y \ge -n+\tilde r$,
 $$
 -||\Phi_n(x,y)||^2+x^2+y^2 = \frac{4\bar r}{\sqrt{1+p_n'(x_m)^2}}(x_mp_n'(x_m)-p_n(x_m)),
 $$
where $\overline r$ is defined in \eqref{r}. The concavity of $p_n$ in $[-x_n-\tilde r,x_n+\tilde r]$ ensures that
$$
x_mp_n'(x_m)-p_n(x_m) \le -p_n(0).
$$
Hence, for $ y \ge -n+\tilde r$, we have
$$
\exp\left(\frac{-||\Phi_n (x,y)||^2}{2}+\frac{x^2+y^2}{2}\right) \le \exp\left(-\frac{2\overline r p_n(0)}{\sqrt{1+p_n'(x_m)^2}}\right)
\le \max\{1,e^{-2rp(0)}\}.
$$
Analogously, when $y<-n+\tilde r$, since, without loss of generality we may assume $-q_n(0)>0$, we get
$$
\exp\left(\frac{-||\Phi_n (x,y)||^2}{2}+\frac{x^2+y^2}{2}\right) \le 1
$$
and the claim \eqref{exp} follows.

Finally, by \eqref{jac}, \eqref{cris} and \eqref{exp} we get
\begin{eqnarray}\label{u2}
&\dint_{\mathbb{R}^2}& \tilde u_n^2 d\gamma_2
\\
&=& \int_{\Omega_n}u_n^2d\gamma_2
+ \int_{\overline \Omega_n \setminus \Omega_n} \tilde u_n^2 d\gamma_2 \notag\\
&= & \int_{\Omega_n}u_n^2d\gamma_2  + \int_{\Omega_n \setminus \underline
\Omega_n} \theta_n^2(\Phi_n(x,y))  u_n^2(x,y) \exp\left(-\frac{
||\Phi_n (x,y)||^2}{2}+\frac{x^2+y^2}{2}\right) |J_{\Phi_n}|d\gamma_2 \notag \\
&\le& C \int_{\Omega_n}u_n^2d\gamma_2,\notag
\end{eqnarray}
while \eqref{jac}, \eqref{energy} and \eqref{u2} imply
\begin{eqnarray*}
&\dint_{\mathbb{R}^2}& |D\tilde u_n|^2 d\gamma_2 \\
&\le& C\int_{\Omega_n}u_n^2d\gamma_2+\int_{\Omega_n}|D
u_n|^2d\gamma_2 + C \int_{\Omega_n} |D u_n|^2 \exp\left(-\frac{||\Phi_n
(x,y)||^2}{2}+\frac{x^2+y^2}{2}\right) |J_{\Phi_n}|d\gamma_2 \\
&\le & C \left[\int_{\Omega_n}u_n^2d\gamma_2+ \int_{\Omega_n} |Du_n|^2d\gamma_2\right].
\end{eqnarray*}
We have proved that
\begin{equation}\label{ext}
||\tilde u_n||_{H^1(\R^2,d\gamma_2)} \le C || u_n||_{H^1(\Omega_n,d\gamma_2)}
\end{equation}
with $C$ dependent only on $\tilde r$.

\begin{remark}
Note that in Step 2 we do not use the symmetry assumption but just the fact that $\partial\Omega$ does intersect the $y$-axis. Therefore, by repeating the same arguments used in Step 2, we can prove that, if $\Omega$ is a convex, $C^2$ domain, satisfying a uniform interior sphere condition such that $\partial\Omega \cap \{x=0\}\neq \emptyset$, then there exists an extension operator $P:H^1(\Omega,d\gamma_2) \to H^1(\R^2,d\gamma_2)$ such that
\begin{equation} \label{ultima}
||Pu||_{H^1(\R^2,d\gamma_2)} \leq C ||u||_{H^1(\Omega,d\gamma_2)}
\end{equation}
with $C$ depending only on the radius of the interior sphere.
 In turn, as done for example in \cite{FP}, from \eqref{ultima} we can derive the compactness of the embedding of $H^1(\Omega,d\gamma_2)$ into $L^2(\Omega,d\gamma_2)$.
\end{remark}

\textbf{Step 3: Passing to the limit.}

Observe that the sequence $\mu_1^{odd}(\Omega_n)$ is bounded from above and
from below by two positive constants. Indeed, using \eqref{estimate}, we
have
\begin{equation}  \label{mu}
1=\mu_1(-\infty,+\infty) \le \mu_1(-x_n-\tilde r,x_n+\tilde r) \le \mu_1^{odd}(\Omega_n) \le \dfrac{
\gamma_2(\Omega_n)}{\displaystyle\int_{\Omega_n} x^2 d\gamma_2} \le \frac{
\gamma_2(\Omega_1)}{\displaystyle\int_{\Omega_1} x^2 d\gamma_2}.
\end{equation}
Thus, up to a subsequence, $\mu_1^{odd}(\Omega_n)$ converge to a number $
\mu>0$.

\noindent Now, let us normalize $u_n$  in such a way that $u_n(x,y)>0$
if $x>0$ and $\displaystyle\int_{\Omega_n}u_n^2d\gamma_2=1$ for every $n$.
 From \eqref{mu} and \eqref{ext} we deduce that the sequence $\tilde u_n$ is
bounded in $H^1(\mathbb{R}^2,d\gamma_2)$. Then, up to a subsequence, $\tilde
u_n$ converges to a function $v$ weakly in $H^1(\Omega,d\gamma_2)$, strongly
in $L^2(\Omega,d\gamma_2)$ and a.e. in $\Omega$. Let $\varphi \in
H^1(\Omega,d\gamma_2)$; then
\begin{equation*}
\int_{\Omega} D\tilde u_n D\varphi d\gamma_2 =\int_{\Omega_n}Du_n D \varphi
d\gamma_2 +\int_{\Omega\setminus\Omega_n} D\tilde u_n D\varphi d\gamma_2
\end{equation*}
and
\begin{equation*}
\int_{\Omega} \tilde u_n \varphi d\gamma_2 =\int_{\Omega_n} u_n \varphi
d\gamma_2 +\int_{\Omega\setminus\Omega_n} \tilde u_n \varphi d\gamma_2.
\end{equation*}
Since
\begin{equation*}
\int_{\Omega\setminus \Omega_n} D\tilde u_n D\varphi d\gamma_2 \le
\left(\int_{\mathbb{R}^2}|D\tilde
u_n|^2d\gamma_2\right)^{1/2}\left(\int_{\Omega\setminus \Omega_n}|D\varphi
|^2d\gamma_2\right)^{1/2}\le C \left(\int_{\Omega\setminus
\Omega_n}|D\varphi |^2d\gamma_2\right)^{1/2}
\end{equation*}
and $\lim_{n \to \infty}\gamma_2(\Omega \setminus \Omega_n)=0$, we have
\begin{equation*}
\lim_{n \to \infty} \int_{\Omega\setminus\Omega_n} D\tilde u_n D\varphi
d\gamma_2=0.
\end{equation*}
Analogously, it holds
\begin{equation*}
\lim_{n \to \infty} \int_{\Omega\setminus\Omega_n} \tilde u_n \varphi
d\gamma_2=0.
\end{equation*}
Then, up to subsequences,
\begin{eqnarray*}
\int_{\Omega} Dv D\varphi &=& \lim_{n \to \infty}\int_{\Omega} D \tilde u_n
D\varphi d\gamma_2 = \lim_{n \to \infty} \int_{\Omega_n}Du_n D \varphi
d\gamma_2 \\
&=& \lim_{n \to \infty} \mu_1^{odd} (\Omega_n) \int_{\Omega_n} u_n \varphi
d\gamma_2 = \mu \lim_{n \to \infty}\int_{\Omega} \tilde u_n \varphi
d\gamma_2 = \mu \int_{\Omega} v \varphi d\gamma_2.
\end{eqnarray*}
On the other hand, $v$ inherits the sign of $u_n$, more precisely $v>0$ if $
x>0$ and $v<0$ if $x<0$. Then $v$ is an eigenfunction corresponding to $
\mu_1^{odd}(\Omega)$. Finally, by applying \eqref{mu}, we have
\begin{equation*}
\mu_1^{odd}(\Omega)=\lim_{n \to \infty} \mu_1^{odd}(\Omega_n) \ge
\mu_1(-a,a)
\end{equation*}
and hence case $a)$  is proved.

The proof of case $b)$ is much more simple, since the boundary of $\Omega$ contains two parallel half-lines. Thus the reflection map through these two straight lines has obviously jacobian 1.
\end{proof}

Finally, let  $S_a=\{(x,y)\in \mathbb{R}^2:\> |x|<a\}$. We observe that
$$
\mu_1^{odd}(S_a)=\mu_1(-a,a)=1+\lambda_1(-a,a)=\mu_1(S_a)+\lambda_1(-a,a).
$$
The following proposition is an immediate consequence of estimate \eqref{estimate}.\begin{proposition}
If $\Omega$ satisfies assumptions of Theorem \ref{T1} or \ref{T2} and $\mu_1(\Omega)=\mu_1^{odd}(\Omega)$, then
$$
\mu_1(\Omega)-\mu_1(S_a) =\mu_1(\Omega)-1\ge \lambda_1(-a,a).
$$
\end{proposition}

\begin{remark}
 As already observed in the Introduction, the equality
 \begin{equation}\label{=}
 \mu_1(\Omega)=\mu_1^{odd}(\Omega)
 \end{equation}
 holds for instance when $\Omega$ is any disk centred at the origin or any square $(-l,l)^2$ ($l>0$).

\noindent Anyway, all the assumptions of Theorem \ref{T1} or \ref{T2} are not enough to guarantee
\eqref{=} to hold. Indeed, denoted by $T=(-1,1)\times (-\infty ,0)$, it is
easy to verify that $2=\mu _{1}(T)<\mu _{1}^{odd}(T)=3$ with corresponding
eigenfunctions $u(x,y)=y^{2}-1$ and $u^{odd}(x,y)=x^{3}-3x$.
\end{remark}
{\sc Acknowledgment.} We would like  to thank professor A. Henrot for useful discussions.

\end{document}